\documentclass[12pt]{article}%

\usepackage{amsfonts}%
\usepackage{amssymb}%
\usepackage{amsmath}%
\usepackage{amsthm}
\usepackage{amssymb}
\usepackage{array}
\usepackage{graphicx}
\usepackage[text={6.0in,8.5in}]{geometry}
\usepackage[colorlinks,citecolor=blue,urlcolor=blue]{hyperref}
\usepackage{thmtools}
\usepackage{tikz}
\usepackage{setspace}
\usepackage{subfig}
\usepackage{nicefrac}
\usepackage{natbib} 
\usepackage{multirow}
\usepackage{epigraph}

\theoremstyle{definition}

\newtheorem{corollary}{Corollary}

\newtheorem{proposition}{Proposition}

\begin{document}

\title{A Characteristic Function for the Primes}

\author{Jesse Aaron Zinn\thanks{Assistant Professor of Economics at the College of Business, Clayton State University.}}
\date{May 2, 2016}
\maketitle

\begin{abstract}
I develop a function that, for any integer $n \geq 2$, takes a value of 1 if $n$ is prime, 0 if $n$ is composite. I also discuss two applications: First, the characteristic function provides a new expression for the prime counting function. Second, the components of the characteristic function point to a new expression that gives the number of distinct prime factors for any integer greater than one.

\end{abstract}

\section{Introduction}
There is an often-cited quote from \cite{zagier1977first} that I believe is so relevant to this work as to bear repeating:
\begin{quote}
	There are two facts about the distribution of prime numbers of
	which I hope to convince you so overwhelmingly that they will be
	permanently engraved in your hearts. The first is that, despite their 
	simple definition and role as the building blocks of the natural numbers, the prime numbers belong to the most arbitrary and ornery objects studied by mathematicians: they grow like weeds among the natural numbers, seeming to obey no other law than that of chance, and nobody can predict where the next one will sprout. The second fact is even more astonishing, for it states just the opposite: that the prime numbers exhibit stunning regularity, that there are laws governing their behavior, and that they obey these laws with almost military precision.
\end{quote} 
In this paper I describe some of the ``laws'' that ensure the ``stunning regularity'' of the primes.

The central result is the development of a characteristic function $\chi_{\mathbb{P}}$ for the set of prime numbers $\mathbb{P}$. This function exploits the cyclical nature of the multiples of any given natural number greater than unity, not unlike how the sieve of Eratosthenes yields primes by ``sifting away'' the multiples of known primes.\footnote{See \cite[pp. 114-116]{dunham1990journey} for a short discussion of Eratosthenes and his sieve.}

I also present, in \S \ref{sec:Applications}, two applications of the characteristic function of the primes: an expression for the prime counting function and another expression for the number of distinct prime factors for any give natural number greater than unity.

\section{The Underlying Framework}

The first step is to find a function that ``indicates'' whether a number is a multiple of another. Such a function is supplied in the following proposition.

\begin{proposition}\label{prop:base}
For any integers $n \geq m \geq 2$, define
\[
E_m(n) := 1 - \prod_{j=1}^{m-1}\frac{\sin^2(\frac{n+j}{m}\pi)}{\sin^2(\frac{j}{m}\pi)}.
\]
This function equals 0 if $n$ is divisible by $m$. Otherwise the function equals 1. 
\end{proposition}  

\begin{proof}
	Suppose $n$ is a multiple of by $m$. Then we can define $c = \frac{n}{m} \in \mathbb{N}$, and we can write
	\[
	E_m(n) = 1 - \prod_{j=1}^{m-1}\frac{\sin^2((c+\frac{j}{m})\pi)}{\sin^2(\frac{j}{m}\pi)}.
	\]
There are two cases to consider: $c$ being either even or odd. If $c$ is even then, since the sine function has period of $2\pi$, we have
\begin{align*}
E_m(n) & = 1 - \prod_{j=1}^{m-1}\frac{\sin^2(\frac{j}{m}\pi)}{\sin^2(\frac{j}{m}\pi)}\\
  & = 1 - 1\\
  & = 0.
\end{align*}
If $c$ is odd then there exists some $q \in \mathbb{N}$ such that $c = 2q + 1$, so
\[
E_m(n) = 1 - \prod_{j=1}^{m-1}\frac{\sin^2((2q + 1+\frac{j}{m})\pi)}{\sin^2(\frac{j}{m}\pi)}
\]
As $\sin(\pi + x) = -\sin x$ for all $x$, it follows that
\[
E_m(n) = 1 - \prod_{j=1}^{m-1} \frac{(-\sin((2q+\frac{j}{m})\pi))^2}{\sin^2(\frac{j}{m}\pi)}
\]
Again use the fact that the sine function has period of $2\pi$ alongside the fact that the square of the negative is the square of the positive:
\begin{align*}
E_m(n) & = 1 - \prod_{j=1}^{m-1}\frac{(-\sin(\frac{j}{m}\pi))^2}{\sin^2(\frac{j}{m}\pi)}\\
& = 1 - 1\\
& = 0.
\end{align*}

Now suppose that $n$ is not a multiple of $m$, then for some $j \in \{1,\ldots,m-1\}$, it must be the case that $n+j$ is a factor of $m$ so one of the $\sin^2(\frac{n+j}{p}\pi)$ terms will be 0, so
\begin{align*}
E_m(n) & = 1 - \prod_{j=1}^{m-1}\frac{\sin^2(\frac{n+j}{m}\pi)}{\sin^2(\frac{j}{m}\pi)}\\
& = 1 - 0\\
& = 1. \qedhere
\end{align*}
\end{proof}

Upon a close reading of the proof for Proposition \ref{prop:base} one can see that $E_m(n)$ will tell whether $n$ is a multiple of $m$ by essentially checking if any of the numbers from $n+1$ to $n+m-1$ are multiples multiples of $m$. If not then $n$ is a multiple of $m$, and if so then $n$ cannot be a multiple of $m$. 

This may seem like an unnecessarily indirect approach. Why not simply check whether $n$ is a multiple of $m$ directly by evaluating $\sin^2(\frac{n}{m}\pi)$? This expression equals zero if $n$ is a multiple of $m$ and some number in $(0,1)$ otherwise. For our purposes, the expression needs to be normalized whenever $n$ is not a multiple of $m$. A problem with such an approach is that the requisite degree of normalization varies from one number to the next. For example, suppose we are checking whether $6$ or $7$ are multiples of $5$. We have $\sin^2(\frac{6}{5}\pi) \approx 0.35$ and $\sin^2(\frac{7}{5}\pi) \approx 0.90$. With $E_m(n)$, we leap this hurdle of variation in normalization by simply multiplying the values of sine functions together and then normalizing with the product of what would have been the individual normalizing factors.

\section{Formulae for Prime Characterization}

If $ab = n$ then either $a \leq \sqrt{n}$ or $b \leq \sqrt{n}$. This fact in conjunction with Proposition \ref{prop:base} implies that if $E_m(n) = 0$ for any integer $m \in [2,\sqrt{n}\,]$ then $n$ is a composite number. We can express this as 
\begin{equation}\label{eq:logic1}
n \textrm{ is composite} \quad \Leftrightarrow \quad \textrm{There exists an integer } m \in [2,\sqrt{n}\,] \textrm{ with } E_m(n) = 0.
\end{equation}
Similarly, if $E_m(n) = 1$ for all integers $m \in [2,\sqrt{n}\,]$ then $n$ has no whole number factors other than 1 and itself, in which case $n$ is prime so\footnote{Some readers may wonder about the cases $n = 2$ and $n = 3$ in which $\sqrt{n} < 2$. In these cases, $[2,\sqrt{n}\,]$ is the empty set, so it is vacuously true that $E_m(n) = 0$ for each integer $m \in [2,\sqrt{n}\,]$. There are no counterexamples in the empty set.}
\begin{equation}\label{eq:logic2}
n \textrm{ is prime} \quad \Leftrightarrow \quad E_m(n) = 1 \textrm{ for all positive integers $m \in [2,\sqrt{n}\,]$.}
\end{equation}
And we have our next result:

\begin{proposition}\label{prop:chi}
For any integer $n \geq 2$,
\begin{equation}\label{eq:chi1}
\begin{aligned}
\chi_{\mathbb{P}}(n) & = \prod_{\substack{m \in [2,\sqrt{n}] \\ m\in \mathbb{N}}} E_m(n)\\
   & = \prod_{\substack{m \in [2,\sqrt{n}] \\ m\in \mathbb{N}}}  \left(1- \prod_{j=1}^{m-1}\frac{\sin^2(\frac{n+j}{m}\pi)}{\sin^2(\frac{j}{m}\pi)} \right).
\end{aligned}
\end{equation}
is a characteristic function of the primes. That is $\chi_{\mathbb{P}}(n) = 1$ if $n$ is prime and $\chi_{\mathbb{P}}(n) = 0$ if $n$ is composite.\footnote{For $n =2$ and $n = 3$, the right-hand side of expression \eqref{eq:chi1} is the empty product. Following the convention that the value of the empty product is one provides the correct results, $\chi_{\mathbb{P}}(2) = \chi_{\mathbb{P}}(3) = 1$.}
\end{proposition}
\begin{proof} 
	This follows directly from expressions \eqref{eq:logic1} and \eqref{eq:logic2}.
\end{proof}
It is possible to lower the number of terms used to calculate $\chi_{\mathbb{P}}(n)$ in expression \eqref{eq:chi1}, particularly if one knows some or all of the primes up to $\sqrt{n}$. Thanks to the fundamental theorem of arithmetic, it is sufficient to include only the terms indexed by prime numbers in the right-hand side of expression \eqref{eq:chi1}, a fact that provides the following corollary.

\begin{corollary}\label{cor:chi}
For any integer $n \geq 2$,
\begin{equation}\label{eq:chi2}
\begin{aligned}
\chi_{\mathbb{P}}(n) & = \prod_{\substack{p \leq \sqrt{n} \\ p\in \mathbb{P}}} E_p(n)\\
  & = \prod_{\substack{p \leq \sqrt{n} \\ p\in \mathbb{P}}} \left(1- \prod_{j=1}^{p-1}\frac{\sin^2(\frac{n+j}{p}\pi)}{\sin^2(\frac{j}{p}\pi)} \right).
\end{aligned}
\end{equation}
\end{corollary}

\section{Applications}\label{sec:Applications}

Having developed a formula for the characteristic function of the primes $\chi_{\mathbb{P}}(n)$, this section highlights two applications. The first is the prime counting function and the second is another formula for the number of distinct prime factors of any given number.

\subsection{The Prime Counting Function}

A natural extension of an indicator function for the primes is the prime counting function $\pi(x)$, which gives the number of primes less than or equal to $x$. It is well known that $\pi(x)  = \sum_{n = 2}^{x} \chi_{\mathbb{P}}(n)$. Thus, from Proposition \ref{prop:chi}, we have
\begin{equation}\label{eq:counting1}
\pi(x) = \sum_{n = 2}^{x} \left(  \prod_{\substack{m \in [2,\sqrt{n}] \\ m\in \mathbb{N}}}  \left(1- \prod_{j=1}^{m-1}\frac{\sin^2(\frac{n+j}{m}\pi)}{\sin^2(\frac{j}{m}\pi)} \right)  \right). 
\end{equation}
Likewise, Corollary \ref{cor:chi} yields
\begin{equation}\label{eq:counting2}
\pi(x) = \sum_{n = 2}^{x} \left( \prod_{\substack{p \leq \sqrt{n} \\ p\in \mathbb{P}}} \left(1- \prod_{j=1}^{p-1}\frac{\sin^2(\frac{n+j}{p}\pi)}{\sin^2(\frac{j}{p}\pi)} \right) \right). 
\end{equation}

These expressions for the prime counting function have several features that other expressions for the prime counting function do not. In particular, expression \eqref{eq:counting1} allows one to count the primes without having knowledge of any of the primes beforehand. Also, both expressions \eqref{eq:counting1} and \eqref{eq:counting2} contain no non-analytical components, no ceiling functions, floor functions, etc. This is in sharp contrast to the formula of Legendre and the derivatives thereof.\footnote{See, for example, \cite[pp. 10-12]{riesel2012prime}.}

\subsection{The Number of Distinct Prime Factors}
A function of some interest in number theory is $\omega(n)$, which gives the number of distinct prime factors of any integer $n \geq 2$.\footnote{A discussion of $\omega(n)$ is contained in \cite[p. 354]{hardy1979introduction}.} The following proposition provides an analytic extression for $\omega(n)$.

\begin{proposition} For any natural number $n = p_1^{\alpha_1} p_2^{\alpha_2} \cdots p_{\omega(n)}^{\alpha_{\omega(n)}}$, where $p_k \in \mathbb{P}$ and $\alpha_k \in \mathbb{N}$ for each $k$, the following is true:
	\[
	\begin{aligned}
	\omega(n) & = \sum_{k=1}^{\omega(n)} (1 - E_{p_k}(n))\\
	 & = \sum_{\substack{p\leq n \\ p\in \mathbb{P}}} \left( \prod_{j=1}^{p-1}\frac{\sin^2(\frac{n+j}{p}\pi)}{\sin^2(\frac{j}{p}\pi)} \right).
	\end{aligned}
	\]
\end{proposition}

\begin{proof}
	From Proposition \ref{prop:base}, $1 - E_{p_k}(n) = 1$ for each $k = 1, 2, \ldots \omega(n)$, which implies
	
	\[
	\sum_{k=1}^{\omega(n)} \left(1 - E_{p_k}(n)\right) = \omega(n).	
	\]
	By the Fundamental Theorem of Arithmetic, $n = p_1^{\alpha_1} p_2^{\alpha_2} \cdots p_{\omega(n)}^{\alpha_{\omega(n)}}$ is the unique prime factorization of $n$, so
	 \begin{align*}
	   \sum_{\substack{p\leq n \\ p\in \mathbb{P}}} \left( \prod_{j=1}^{p-1}\frac{\sin^2(\frac{n+j}{p}\pi)}{\sin^2(\frac{j}{p}\pi)} \right)
	    & = \sum_{k=1}^{\omega(n)} \left(\prod_{j=1}^{p_k-1}  \frac{\sin^2(\frac{n+j}{p_k}\pi)}{\sin^2(\frac{j}{p_k}\pi)}\right)\\
	    & = \sum_{k=1}^{\omega(n)} \left(1 - E_{p_k}(n)\right).\\
	    & = \omega(n). \qedhere
	 \end{align*}	 
\end{proof}

\bibliography{C:/bibtex/library}
\end{document}